\documentclass [a4paper, 10pt] {article}
\usepackage[cp1251]{inputenc}
\usepackage{amsmath,amssymb,longtable}
\usepackage[english]{babel}
\usepackage[final]{graphicx}
\usepackage{latexsym}
\usepackage{amsthm}
\usepackage{xcolor}

\usepackage{epstopdf}
\DeclareGraphicsExtensions{.pdf,.eps,.png,.jpg,.mps}

\newcommand{\B}{\mathbb{B}}
\newcommand{\Z}{\mathbb{Z}}
\newcommand{\R}{\mathbb{R}}

\newcommand{\PP}{\mathbb{P}}
\newcommand{\rem}{\vartriangleright}
\newcommand{\ple}{\preceq}

\newtheorem{cor}{Corollary}
\newtheorem{observation}{Observation}
\newtheorem{lemma}{Lemma}
\newtheorem{defi}{Definition}

\global\emergencystretch = 80 pt

\newcommand{\highlight}[1]{#1}
\newcommand{\highlightt}[1]{#1}
\newcommand{\highlighttt}[1]{#1}

\begin{document}
\begin{center}{\Large \bf
Minimal proper non-IRUP instances of the one-dimensional Cutting Stock Problem
}
 
\bigskip

Vadim M. Kartak$^{1)}$, Artem V. Ripatti$^{2)}$,

Guntram Scheithauer$^{3)}$ and Sascha Kurz$^{4)}$

\medskip

{\small
$^{1)}$Bashkir State Pedagogical University, Russia\\
$^{2)}$Ufa State Aviation Technical University, Russia\\
$^{3)}$Technische Universit\"at Dresden, Germany\\
$^{4)}$Universit\"at Bayreuth, Germany
}

\today
\end{center}

\textbf{Abstract}: 
We consider the well-known one dimensional cutting stock problem (1CSP). Based on the pattern structure of the classical
ILP formulation of Gilmore and Gomory, we can decompose the infinite set of 1CSP instances, with a fixed demand $n$, into 
a finite number of equivalence classes. \highlight{We show up a strong relation to weighted simple games.} Studying the integer 
round-up property we computationally show that all 1CSP instances with \highlight{$n\le 9$} are proper IRUP, while we give 
\highlightt{examples of a proper non-IRUP instances} with \highlight{$n=10$}. \highlight{A gap larger than $1$ occurs for $n=11$.} The worst known 
gap is \highlight{raised} from $1.003$ to $1.0625$. The used algorithmic approaches are based on exhaustive enumeration and integer linear 
programming. Additionally we give some theoretical bounds showing that all 1CSP instances with some specific parameters have the proper IRUP.

\medskip


\medskip

\section{Introduction}

Baum and Trotter \cite{baum1981integer} have introduced the notation of having the integer round-up property (IRUP) for integer 
linear minimization problems (ILPs), stating that \highlight{rounding up} the optimal value \highlight{of} its LP relaxation yields an 
upper bound. Here we study the one-dimensional cutting stock problem (1CSP) with respect to the IRUP. The classical ILP formulation for 
the cutting stock \highlight{problem} by Gilmore and Gomory is based on so-called cutting patterns \cite{gilmore1961linear}. Using this 
formulation, Marcotte \cite{marcotte1985cutting} has shown that certain subclasses of 
cutting stock problems have the IRUP, while she later showed that there are instances of 1CSP having a gap of exactly 
$1$ \cite{marcotte1986instance}. The first example with gap larger than $1$ was given in \cite{fieldhouse1990duality}. 
Subsequently, the gap was increased to $\frac{6}{5}$ \cite{rietz2008large,scheithauer1992gap,scheithauer1997theoretical}, i.e., no example 
with a gap of at least $2$ is currently known. Indeed, the authors of \cite{scheithauer1995modified} have conjectured that the gap is always 
below $2$ -- called the MIRUP property, which is one of the most important theoretical issues about 1CSP, see also \cite{eisenbrand2013bin}. 
Practical experience shows  that the typical gap is rather small \cite{scheithauer1995branch}. An algorithm for verifying when an 1CSP 
instance does not have IRUP is presented in \cite{kartak2004sufficient}.

Dropping some cutting patterns from the ILP formulation of \cite{gilmore1961linear}, the authors of \cite{nitsche1999tighter} \highlight{introduced}
the \textit{proper relaxation}, which is at most as worse as the LP relaxation. We call the difference of the optimal value and the optimal 
value of the corresponding proper relaxation the proper gap. If the proper gap is at least $1$, we speak of a proper non-IRUP instance. The
\highlight{currently} largest known proper gap of 1CSP is given by $1.003$ \cite{friendly_bin_packing,phd_diaz}.

The paper is organized as follows. In Section~\ref{sec_notation} we introduce the basic notation. The concept of partitioning 
the infinite set of 1CSP with fixed demand into a finite set of equivalence classes is described in Section~\ref{sec_equivalence}. 
The relation to the discrete structure of weighted simple games, from cooperative game theory, is the topic of Section~\ref{sec_simple_games}. 
In Section~\ref{sec_enumeration} we develop an exhaustive enumeration algorithm for the generation of all equivalence classes of 
1CSP instances. We proceed with some theoretical bounds on the proper gap of 1CSP instances in Section~\ref{sec_bounds}. \highlight{Based} on the
integer linear programming approaches in Section~\ref{sec_ILP} we present computational results in Section~\ref{sec_results}. We draw 
a conclusion in Section~\ref{sec_conclusion}.

\section{Basic notation}
\label{sec_notation}
Assume that one-dimensional material objects like, e.g., paper reels or wooden rods of a given length $L\in\mathbb{R}_{>0}$ are
cut into smaller pieces of lengths $l_1,\dots,l_m\in\mathbb{R}_{>0}$ in order to fulfill the order demands $b_1,\dots,b_{\highlightt{m}} \in \mathbb{Z}_{>0}$. 
The question for the minimum needed total amount of stock material or, equivalently, the minimization of waste, is the famous 1CSP. Using 
the abbreviations $l=(l_1,\dots,l_m)^T$ and $b=(b_1,\dots,b_m)^T$ we denote an instance of 1CSP by $E=(m,L,l,b)$. W.l.o.g.\ we assume 
$0<l_1\le\dots\le l_m\le L$ in the following.

The cutting patterns, mentioned in the introduction, are formalized as vectors $a=(a_1,\dots,a_m)^{\top}\in\mathbb{Z}_{\ge 0}^m$. 
We call a \textit{pattern} (of $E$) \textit{feasible} if $l^{\top}a\le L$. By $P^f(E):=\left\{a\,:\,l^{\top}a\le L,\,a\in\mathbb{Z}_{\ge 0}^m\right\}$ 
we  denote the set of all feasible patterns. Additionally, we call a pattern \textit{proper}, if it is feasible and $a_i\le b_i$ for all $1\le i\le m$. 
By $P^p(E):=\left\{a\,:\,l^{\top}a\le L,\,a\in\mathbb{Z}_{\ge 0}^m,\, a_i\le b_{\highlight{i}},\, 1\le i\le m\right\}$ we denote the set of all proper patterns.

Given a set of patterns $P=\{a^1,\dots,a^r\}$ (of $E$), let $A(P)$ denote the concatenation of the pattern vectors $a^i$. With this 
we can define
$$
  z_D(P,E) := \sum_{i=1}^{\highlightt{r}} x_i \to \min \; \text{ subject to } \; A(P)x = b, \; x \in \Z^{\highlightt{r}}_{\ge 0}\quad \text{and}
$$
$$
  z_C(P,E) := \sum_{i=1}^{\highlightt{r}} x_i \to \min \; \text{ subject to } \; A(P)x = b, \; x \in \R^{\highlightt{r}}_{\ge 0}.
$$
Choosing $P=P^f(E)$ we obtain the mentioned ILP formulation for 1CSP of Gilmore and Gomory and its continuous relaxation. 
As abbreviations we use $z_D^f(E):=z_D(P^f(E),E)$, $z_C^f(E):=z_C(P^f(E),E)$, and $\Delta(E):=z_D^f(E)-z_C^f(E)$, where the
later is called the \textit{gap of instance $E$}. So an instance $E$ has \textit{IRUP} if $\Delta(E)<1$ and is called 
\textit{non-IRUP instance} otherwise. \textit{MIRUP instances} are those with $\Delta(E)<2$.

Choosing $P=P^p(E)$ we obtain the proper relaxation with optimal value $z_C^p(E):=z_C(P^p(E),E)$. Since 
$z_D(P^p(E),E)=z_D(P^f(E),E)$, we call $\Delta_p(E):=z_D^f(E)-z_C^p(E)$ the \textit{proper gap of instance $E$}. Similarly, an 
instance $E$ is a \textit{proper IRUP instance} if $\Delta_p(E)<1$ and a \textit{\highlight{proper} non-IRUP instance} otherwise.

Due to $\Delta_p(E) \le \Delta(E)$ proper non-IRUP instances are also non-IRUP instances. The 
converse is not true as shown by $E = \left(3, 30, (2,3,5)^{\top}, (1,2,4)^{\top}\right)$ with $\Delta(E) = 31/30$ and
$\Delta_p(E) = 4/5$.

\section{Equivalence of 1CSP instances}
\label{sec_equivalence}

Given an 1CSP instance $E=(m,L,l,b)$ with a demand of $n=\sum_{i=1}^m b_i$, we can easily transform it into
an instance $\bar{E}=(n,L,l',b')$ with $b_i'=1$ for all $1\le i\le n$ by taking $b_j$ copies of length $l_j$ for each 
$1\le j\le m$. We can easily check that this has no effect on the stated (I)LPs, i.e., we have $z_D^f(E)=z_D^f(\bar{E})$, 
$z_C^f(E)=z_C^f(\bar{E})$, and $z_C^p(E)=z_C^p(\bar{E})$. Thus we assume $b_i=1$ and abbreviate 1CSP instances by $E=(n,L,l)$ in 
the following. Especially we have $P^p(E)\in\mathbb{B}^n$, where $\mathbb{B}=\{0,1\}$.

We remark that, using this modification, the 1CSP becomes equivalent to the Bin Packing Problem (BPP), where $n$ 
items of size $l_i$ have to be packed into as few as possible identical bins of capacity $L$. So our results also hold for the 
BPP and indeed some part of the related literature considers the BPP instead of the 1CSP. 
\highlightt{The continuous  relaxation of BPP is also known as the Fractional Bin Packing Problem, cf. \cite{friendly_bin_packing}.}

Since the set partitioning formulation of Gilmore and Gomory and its proper relaxation actually do not depend directly 
on the parameters $L$ and $l_i$, we partition the set of 1CSP instances for a fixed demand $n$ according to their set of proper 
patterns.

\begin{defi}
  \label{def_pattern_equivalence}
  1CSP instances $E$ and $\bar{E}$ are \textit{pattern-equivalent} if $P^p(E) = P^p(\bar{E})$.
\end{defi}

Since $z_C^p(E)=z_C^p(\bar{E})$, $z_D^f(E)=z_D^f(\bar{E})$, and $\Delta_p(E)=\Delta_p(\bar{E})$ 
for pattern-equivalent instances $E$ and $\bar{E}$, we can restrict ourselves onto the set of equivalence classes
$$\mathbb{P}_n^p:=\left\{P^p(E)\,:\,E=(n,L,l),\,L\in\mathbb{R}_{>0},\,l\in\mathbb{R}_{>0}^n\right\}.
\footnote{We remark that the authors of \cite{rietz2002families} have introduced a finer equivalence relation, called full 
pattern-equivalence \highlight{and based on $P^f(E) = P^f(\bar{E})$}, which is needed if also $z_C^f(E)$ and $\Delta(E)$ 
should be preserved. \highlight{For example, the instances $E = (6, 30, (6, 6, 10, 10, 11, 15)^{\top})$
and $\bar{E} = (6, 10000, (2000, 2000, 3001, 3250, 3750, 5000)^{\top})$ are proper pattern-equivalent but not full 
pattern-equivalent, because $P^f(\bar{E})$ contains pattern $(0,0,2,0,1,0)$ but $P^f(E)$ does not.}}$$

Obviously, $\left|\mathbb{P}_n^p\right|\le 2^{\left|\mathbb{B}^n\right|}=2^{2^n}$ is finite, but not all 
subsets of $\mathbb{B}^n$ can be attained as proper patterns of an 1CSP instance.

\begin{lemma}
  \label{lemma_ineq}
  Given $P\subseteq\mathbb{B}^n$\highlightt{,}
 an 1CSP instance $E=(n,L,l)$ with $P^p(E)=P$ exists iff the following 
  system of linear inequalities contains a solution:
  \begin{eqnarray}
    1 \le l_1 \le  \dots \le l_n \le L, \label{lp_ineqs}\\
    \sum_{i=1}^{\highlightt{n}} l_i a_i  \le L  && \forall a \in P,\nonumber\\ 
    \sum_{i=1}^{\highlightt{n}} l_i a_i  \ge L+1 && \forall a \in \B^n \setminus P,\nonumber\\
    l_1, l_2, \dots l_n, L \in\R_{\ge 0}.\nonumber
  \end{eqnarray}
\end{lemma}
\begin{proof}
  Let $E=(n,L,l)$ be an 1CSP instance with $P = P^b(E)$. Due to definition we have $0<l_1\le\dots\le l_n\le L$.
  By multiplying the $l_i$ and $L$ with a suitable positive factor, we can ensure $l_1\ge 1$. Similarly, we 
  have $\sum_{i=1}^{\highlightt{n}} l_i a_i\highlight{\le L}$ for all $a\in P$ and $\sum_{i=1}^{\highlightt{n}} l_i a_i  > L$ for all $a \in \B^n \setminus P$
  by definition, so that multiplying the variables with a suitable positive factor ensures the validy of all
  constraints.

  If $L,l$ satisfy the inequalities~(\ref{lp_ineqs}), then $E=(n,L,l)$ is an example of the demanded instance.
\end{proof}

We remark that we can additionally require that $L$ and the $l_i$ are positive integers and indeed we will 
use only integers in our subsequent examples of 1CSP instances.

The parameters $l_i$ and $L$ of Lemma~\ref{lemma_ineq} have the following nice geometric interpretation. 
The hyperplane defined by $\sum_{i=1}^n l_i x_i = L$ perfectly separates the set of feasible patterns $P^f(E)=P^p(E)$ 
and the set of non-feasible patterns $\B^n\setminus P^p(E)$ within the unit-hypercube.  In Figure~\ref{all_classes}
we have depicted all five equivalence classes for $n=3$, where the feasible patterns are marked by filled black circles.

\begin{figure}[ht]
  \begin{center}
    \includegraphics[scale=0.67]{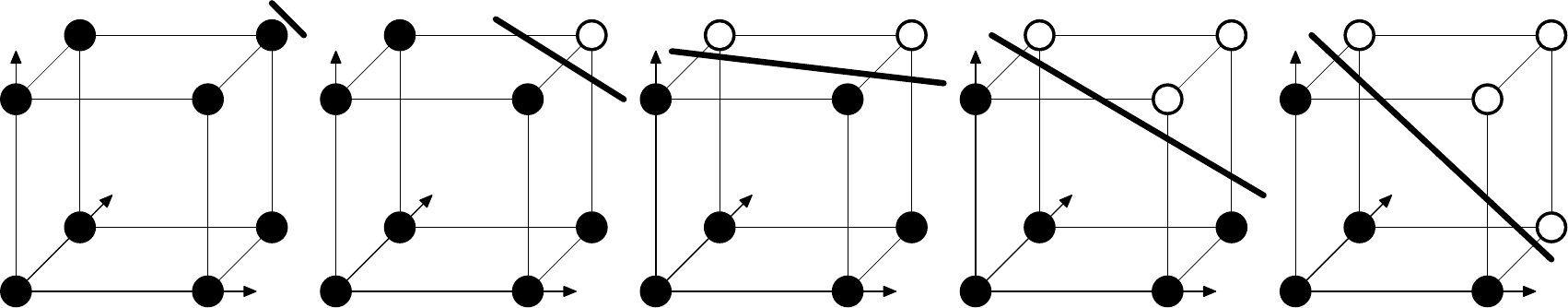}
    \caption{\label{all_classes}All equivalence classes of $\mathbb{P}_3^p$}
  \end{center}
\end{figure}

So a first, simple but finite, algorithm to determine the maximum $\Delta_p(E)$ for given demand $n$, is to 
loop over all equivalence classes in $\mathbb{P}_n^p$ and to compute the respective $\Delta_p(E)$. Of course 
this is possible for rather small $n$ only. Since it is of different interest to explicitly construct 
a complete system of representatives of $\mathbb{P}_n^p$, we present an enumeration algorithm in Section~\ref{sec_enumeration}
before we proceed with ILP approaches in Section~\ref{sec_ILP}. 
\highlightt{Prior to} that we relate our discrete structures 
with another stream of literature in the context of cooperative game theory.

\section{Relation of 1CSP instances to weighted simple games}
\label{sec_simple_games}
In cooperative game theory a \textit{simple game} on $n$ voters is defined as a mapping $v:\B^n\rightarrow\B$ satisfying
$v(\mathbf{0})=0$, $v(\mathbf{1})=1$, and $v(a)\le v(b)$ for all $a,b\in \B^n$ with $a\le b$, i.e., $a_i\le b_i$ for all $1\le i\le n$ 
(cf. \cite{taylor1999simple}). A vector\footnote{Mostly one speaks of subsets $S\subseteq\{1,\dots,n\}$, called coalitions, in the corresponding literature. The vectors we use here correspond to the incidence vectors of 
those sets.} $a\in \B^n$ with $v(a)=1$ is called \textit{winning} and \textit{losing} otherwise.
Each simple game is uniquely characterized by either its set of winning or \highlightt{its} set of losing vectors. A simple game $v$ is 
called \textit{weighted} if \highlight{there} exist weights $w\in\mathbb{R}_{\ge 0}^n$ and a quota $q\in\mathbb{R}_{>0}$ such that
$v(a)=1$ iff $a^{\top}w\ge q$. W.l.o.g.\ we can assume that the quota and the weights are positive integers with 
$1\le w_1\le\dots\le w_n\le q$ and $q\ge 2$.

Given a weighted simple game $v$ represented by weights $w\in\mathbb{Z}_{>0}$ and a quota 
$q\in\mathbb{Z}_{\highlight{\ge 2}}$, we can set $L=q-1$ and $l_i=w_i$ for all $1\le i\le n$. If additionally 
all unit-vectors are losing in $v$, then we have $l_i=w_i\le q-1=L$, i.e., $E=(n,L,l)$ is 
an 1CSP instance, where the losing vectors correspond to the feasible patterns.

For the other direction let $E=(n,L,l)$ be an 1CSP instance with $l\in \mathbb{Z}_{>0}^n$ and 
$L\in\Z_{>0}$. If additionally the all-one vector $\mathbf{1}$ is a non-feasible pattern, then setting $q=L+1$ 
and $w_i=l_i$ for all $1\le i\le n$ yields a weighted simple game $v$.

\section{Enumeration of all pattern-equivalent classes of 1CSP instance}
\label{sec_enumeration}

An equivalence class of \highlight{an} 1CSP instance $E$ is uniquely described by its set $P=P^p(E)\subseteq \B^n$ of feasible patterns.
So we have to enumerate all possible choices for $P$ and subsequently decide which pattern is feasible and which is not.
We observe that infeasibility in Lemma~\ref{lemma_ineq} may happen using only a proper subset of the 
inequalities (\ref{lp_ineqs}).

\begin{lemma}
  \label{lemma_partial_ineq}
  Given two disjoint subsets \highlightt{$P_{\le }$} and $P_>$ of $\B^n$. If
  \begin{eqnarray}
    1 \le l_1 \le  \dots \le l_n \le L, \label{lp_partial}\\
    \sum_{i=1}^{\highlightt{n}} l_i a_i  \le L  && \forall a \in P_{\le},\nonumber\\ 
    \sum_{i=1}^{\highlightt{n}} l_i a_i  \ge L+1 && \forall a \in P_{>},\nonumber\\
    l_1, l_2, \dots l_n, L \in\R_{\ge 0}.\nonumber
  \end{eqnarray}
  does not have a solution, then there can not exist an 1CSP instance $E=(n,L,l)$ with 
  $P_{\le}\subseteq P^p(E)\subseteq \B^n\setminus P_{>}$.
\end{lemma}

Next we observe, that some inequalities of (\ref{lp_ineqs}) and (\ref{lp_partial}) may be dominated by
others. For $a\le b$, i.e, $a_i\le b_i$ for all $1\le i\le n$, we clearly have $l^{\top}a\le l^{\top}b$ due 
to $l\ge 0$. Using the special ordering $l_1\le\dots\le l_n$, we can even uncover more dominated inequalities. 
To this end we introduce the following binary relation.

\begin{defi}
  For $a,b\in\B^n$ we write $a \ple b$ iff $\sum\limits_{i=j}^n a_i\le \sum\limits_{i=j}^n b_i$ for all $1\le j\le n$.
\end{defi}

We say that $a$ is \textit{dominated} by $b$. In the context of simple games the relation $\ple$, using 
the reverse ordering of coordinates, is used to define the class of so-called \highlight{complete} simple games, which is a subclass 
of weighted simple games, see \cite{isbell1956class,taylor1999simple}. So the following results are well known 
in a different context \highlight{and} we mention only the facts that we are explicitly using in this paper.

\begin{lemma}
  Let $a,b\in \B^n$ with $a\ple b$. For $l_1\le\dots\le l_n$ we have $l^{\top}a\le l^{\top}b$.
\end{lemma}
\begin{proof}
  Setting $l_j=\sum\limits_{i=1}^j k_i$, the $k_i\ge 0$ are uniquely defined and we have
  $$
    l^{\top}a=\sum_{j=1}^n
\highlightt{\left( k_j\cdot\sum_{i=j}^n a_i \right)}
\le \sum_{j=1}^n \highlightt{\left( k_j\cdot\sum_{i=j}^n  b_i \right)} =l^{\top}b.
  $$
\end{proof}

\begin{cor}
  Let $\mathbf{0}\le a\ple b\le \mathbf{1}$. If $b\in P^p(E)$, then $a\in P^p(E)$. If $a\notin P^p(E)$, then 
  $b\notin P^p(E)$.
\end{cor}

\begin{lemma}
  $\B^n$ is a partially ordered set under $\ple$.
\end{lemma}

\begin{observation}
  $\{a\in\B^n\,:\,\Vert a\Vert_1 \highlightt{ \le 1} \}\subseteq P^p(E)$ for all 1CSP instances $E=(n,L,l)$ due to $l_i\le L$.
\end{observation}

With those ingredients we can state the following enumeration algorithm.

\begin{longtable}[l]{rl}
\multicolumn{2}{l}{\texttt{MainProcedure}()} \\
1 & $P_{\le} \gets \{ a \in \B^{\highlightt{n}} \,:\, \Vert a\Vert_1=1 \}$ \\
2 & $P_{>} \gets \emptyset$ \\
3 & $P_u \gets \B^{\highlightt{n}} \; \setminus \;  
\{ a \in \B^{\highlightt{n}} \,:\, \Vert a\Vert_1 \le 1 \}$ \\
4 & RecursiveProcedure($P_{\le}, \; P_{>}, \; P_u$) 
\end{longtable}

\begin{longtable}[l]{rl}
\multicolumn{2}{l}{\texttt{RecursiveProcedure}($P_{\le}$,$P_{>}$,$P_u$)} \\
1 & if system~(\ref{lemma_partial_ineq}) has no solution for $P_{\le}$ and $P_{>}$ \\
2 & $\qquad$ return \\
3 & if $P_u = \emptyset$ $\qquad$ $\rem$ we have found a new  equivalence class  \\
4 & $\qquad$ save $\left\{a\in\B^n\,:\,\exists b\in P_{\le}:\, a\ple b\right\}$ \\
5 & $\qquad$ return \\
6 & choose some pattern $a \in P_u$ \qquad $\rem$ no matter which one \\
7 & $P_{\le}' \gets P_{\le}$\\
8 & remove all patterns from $P_{\le}'$ which are dominated by $a$\\
9 & $P_u'\gets P_u\setminus\left\{b\in \B^n\,:\,b\ple a\right\}$\\
10 & RecursiveProcedure($P_{\le}'\cup\{a\}, \; P_{>}, \; P_u'$) \\
11 & $P_{>}' \gets P_{>}$\\
12 & remove all patterns from $P_{>}'$ which dominate $a$\\
13 & $P_u'\gets P_u\setminus\left\{b\in \B^n\,:\,a\ple b\right\}$\\
14 & RecursiveProcedure($P_{\le}, \; P_{>}'\cup\{a\}, \; P_u'$)
\end{longtable}

\highlight{Here  $P_{\le}$ denotes a subset of the patterns that the algorithm has already classified as 
being feasible. Similarly, $P_{>}$ denotes a subset of the patterns that the algorithm has already classified as 
being non-feasible. Since patterns with a single element have to be feasible by definition, we can initialize 
as done in \texttt{MainProcedure}.}

We know that a pattern $a$ is feasible, if there exists a \highlight{pattern} $b\in P_{\le}$ with $a\ple b$ 
\highlight{-- note that $b$ is feasible}. Similarly, we know that a pattern 
\highlightt{$a$} is non-feasible, if there exists a 
\highlight{pattern} $b\in P_{>}$ with $b\ple a$ \highlight{-- note that $b$ is non-feasible}. \highlight{We remark
that all unclassified patterns are pooled in $P_u$.}

\highlight{In order to save computation time within the check of Inequality system~(\ref{lp_partial}), we try to remove 
as many patterns as possible from $P_{\le}$ and $P_{>}$ in lines 8 and 12 of \texttt{RecursiveProcedure}.}
Since $\B^n$ is a partially ordered set under $\ple$, the constructed 
sets $P_{\le}$ and $P_{>}$ are \highlight{indeed} minimal \highlight{in every iteration}. For $n=9$ this \highlight{approach} 
reduces \highlight{the computation time, due to a decreased number of inequalities in ~(\ref{lp_partial}),} by 
\highlight{a} factor of roughly $50$.

The dominance relation $\ple$ can clearly be checked in $O(n)$. Since those comparisons occur quite often it is
beneficial to compute and store them once for all pairs of patterns. Some comparisons can additionally be avoided
by using:

\begin{observation} 
  For $\mathbf{0}\le a\ple b\le \mathbf{1}$ we have $\operatorname{num}(a)\le \operatorname{num}(b)$, where
  $\operatorname{num}(a):=\sum_{i=1}^{n}=2^{i-1}a_i$.
\end{observation}

The converse is generally not true, i.e., $\operatorname{num}(a)\le \operatorname{num}(b)$ implies either $a\ple b$ 
or $a$ and $b$ are incomparable. In Figure~\ref{dom_rel} we have depicted the dominance relation, where patterns are 
ordered by the $\operatorname{num}()$ function. Black squares represent the cases $a \ple b$; white ones the cases $a \not\ple b$.

\begin{figure}[ht]
  \begin{center}
    \includegraphics[scale=0.8]{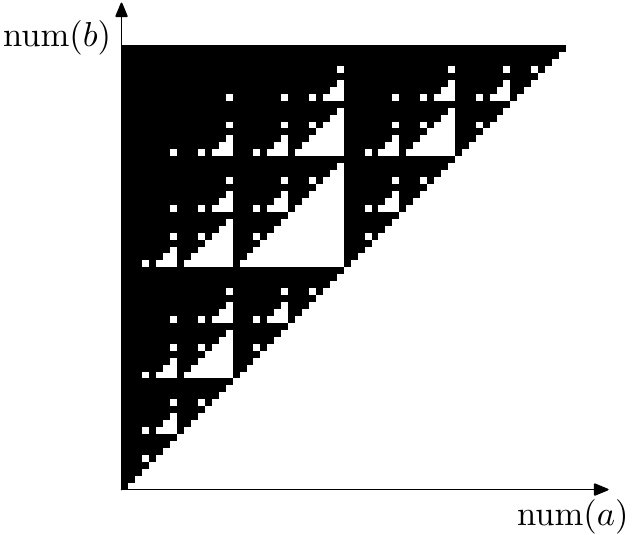}
    \caption{\label{dom_rel}Illustration of the  dominance relation}
  \end{center}
\end{figure} 

\section{Bounds for $\mathbf{\Delta_p(E)}$}
\label{sec_bounds}

Obviously we have $0\le z_C^p(E)\le z_D^f(E)\le n$ for each 1CSP instance $E=(n,L,l)$. The cases with 
$z_D^f(E)=1$ can be completely classified:

\begin{lemma}
  For an 1CSP instance $E=(n,L,l)$ we have\\[-4mm]
  $$
    z_D^f(E)=1\quad\Longleftrightarrow\quad \mathbf{1}\in P^p(E)
    \quad\Longleftrightarrow\quad P^p(E)=\B^n
    \quad\Longleftrightarrow\quad \sum_{i=1}^n l_i\le L.
  $$
\end{lemma}

\begin{cor}
  If $z_D^f(E)=1$, then we have $z_C^p(E)=1$ and $\Delta_p(E)=0$.
\end{cor}

Also the cases where $z_D^f(E)=2$ can be characterized completely:

\begin{lemma}
  \highlight{We have $z_{\highlightt{D}}^p(E)>2$ if and only if $\{a,\mathbf{1}-a\}\not\subseteq P^p(E)$ for all $a\in\B^n$.}
\end{lemma}
\begin{proof}
  \highlight{Choosing $a=\mathbf{1}$ we conclude $\mathbf{1}\notin P^p(E)$, since $\mathbf{0}\in P^p(E)$. Thus we can assume} 
  $z_D^f(E)>1$. We have $z_D^f(E)=2$ iff there exist feasible patterns $a,b\in P^p(E)$ with $a+b\ge\mathbf{1}$. Thus 
  $b \highlightt{=} \mathbf{1}-a\in P^p(E)$.
\end{proof}

We remark that simple games, where not both coalition vectors $a$ and $\mathbf{1}-a$ can be losing, are called \textit{proper}.

The optimal solution of the proper relaxation is \highlight{given} by non-negative reel multipliers $\gamma_a$ satisfying
\begin{equation}
  \sum_{a\in P^p(E)} \gamma_a\cdot a=\mathbf{1}
  \quad\text{and}\quad
  z_C^p(E)=\sum_{a\in P^p(E)} \gamma_a. \label{eq_multipliers}
\end{equation}

\begin{lemma}
  $z_C^p(E)\ge 1$ \highlightt{for any instance $E$ of 1CSP.}
\end{lemma}
\begin{proof}
  From~(\ref{eq_multipliers}) we conclude
  $$
    n\cdot\sum_{a\in P^b(E)} \gamma_a=\sum_{a\in P^p(E)} \gamma_a \cdot n \ge \sum_{a\in P^p(E)} \gamma_a\cdot\Vert a\Vert_1
    =\Vert \mathbf{1}\Vert_1=n. 
  $$\\[-8mm]
\end{proof}

The above proof can be slightly tightened if $\mathbf{1}\notin P^p(E)$.

\begin{lemma}
  If $z_D^f(E)=2$, then $z_C^p(E)\ge \frac{n}{n-1}$ and $\Delta_p(E)\le \frac{n-2}{n-1}$.
\end{lemma}
\begin{proof}
  Since $z_D^f(E)\neq 1$ we have $\mathbf{1}\notin P^p(E)$ so that $\Vert a\Vert_1\le n-1$ for 
  all patterns $a\in P^p(E)$. Combining this with~(\ref{eq_multipliers}) yields
  $$
    (n-1)\cdot\sum_{a\in P^p(E)} \gamma_a=\sum_{a\in P^p(E)} \gamma_a \cdot (n-1) \ge \sum_{a\in P^p(E)} 
    \gamma_a\cdot\Vert a\Vert_1=\Vert \mathbf{1}\Vert_1=n. 
  $$ 
  Thus we have $z_C^p(E)\ge \frac{n}{n-1}$ and $\Delta_p(E)=z_D^f(E)-z_C^p(E)\le \frac{n-2}{n-1}$.
\end{proof}

\begin{lemma}
  If $z_D^f(E)>2$, then 
\highlightt{$z_C^p(E)>2$ and}
$\Delta_p(E)<z_D^f(E)-2$.
\end{lemma}
\begin{proof}
  W.l.o.g.\ we assume that the parameters $L$ and $l_i$ of $E=(n,L,l)$ are integers and that there exists a feasible
  pattern $a\in P^p(E)$ with $l^{\top}a=L$, since we may otherwise decrease $L$ to obtain an equivalent 
  representation with smaller $L$. From $z_D^f(E)>2$ we conclude $\mathbf{1}-a\notin P^p(E)$ so that $l^{\top}(\mathbf{1}-\highlightt{a})>L$.
  Thus  $L<\highlightt{\frac12 {l^{\top}\mathbf{1}}}$.
 Multiplying equation~(\ref{eq_multipliers}) with vector $l$ gives
  $$
    \frac{l^{\top}\mathbf{1}}{2}\sum_{a\in P^p(E)} \gamma_a> \sum_{a\in P^p(E)} \gamma_a\cdot a^{\top}l=\mathbf{1}^{\top}l,
  $$
  so that $z_C^p(E)=\sum_{a\in P^p(E)} \highlightt{ \gamma_a}>2$.
\end{proof}

\begin{cor}
  For $z_D^f(E)=3$ we have $\Delta_p(E)<1$.
\end{cor}

For $z_D^f(E)=4$ we can \highlightt{also} conclude $z_C^p(E)>2$ and $\Delta_p(E)<2$ from $z_D^f(E)\le\frac{4}{3}\cdot\left\lceil z_C^p(E)\right\rceil$, 
see \cite{chan1998worst} for the later relative bound.

\highlight{We remark that the instances with $z_D^p(E)\in \{n-1,n\}$ can be easily characterized. Since their number 
is in $O(n)$, we
abstain from stating the details and provide exemplary enumeration results for $n=8$ in Table~\ref{tab2}.}

\begin{table}[htp]
  \begin{center}
    \caption{\highlight{Number of equivalence classes for $n=8$ and a given $z_D^p(E)$-value\label{tab2}}}
      
    \smallskip
 
    \begin{tabular}{rrrrrrrrrr}
      $z_D^p(E)$ & 1 & 2 & 3 & 4 & 5 & 6 & 7 & 8 \\   
      \#         & 1 & 1363847 & 1277944 & 56895 & 1992 & 103 & 8 & 1
    \end{tabular}
  \end{center}
\end{table}

\section{Integer linear programming approaches}
\label{sec_ILP}

Assume that we are not interested in all equivalence classes of 1CSP but only in those with $\Delta_p(E)\ge \delta$ for 
some parameter $\delta\ge 0$. For the search for proper non-IRUP instances we may set $\delta=1$ and for the search of 
the largest possible $\Delta_p(E)$ for a given demand $n$ we may update $\delta$ during a search algorithm. In the 
following subsections we present two algorithmic approaches.

\subsection{A tailored branch-and-bound algorithm}
\label{subsec_bandb}
We can easily convert the enumeration algorithm from Section~\ref{sec_enumeration} into \highlightt{a} branch-and-bound algorithm 
with some additional cuts. To this end we state:

\begin{lemma}
  Let $E=(n,L,l)$ be an 1CSP instance and $U,V$ be two subsets of $\B^n$ with $V\subseteq P^p(E)\subseteq U$. 
  With this we have $z_C(U,E)\le Z_C(V,E)$ and $z_D(U,E)\le Z_D(V,E)$.
\end{lemma}
\begin{proof}
  \highlight{Each feasible solution for pattern set $V$, i.e., each vector $x$ with $A(V)x=\mathbf{1}$, can be extended 
  to a feasible solution for pattern set $U$ by inserting zeros for all patterns in $U\backslash V$.}
\end{proof}

\begin{cor}
  \label{cor_cut}
  If $V\subseteq P^p(E)\subseteq U\subseteq \B^n$ for an 1CSP instance $E=(n,L,l)$, then 
  $\Delta_p(E)\le z_D(U,E)-z_C(V,E)$.
\end{cor}

Our first modification
of the enumeration algorithm is the extension of the lines 1 and 2 in \texttt{RecursiveProcedure} with the check from Corollary~\ref{cor_cut}.

Depending on the chosen value of $\delta$ we can also utilize some of the bounds from Section~\ref{sec_bounds} to start the algorithm 
with a non-empty set $P_{>}$. For, e.g., $\delta\ge 1$ we know $z_D^f(E)>3$ so that we can set $P_{>}=\left\{a\in\B^n\,:\,\Vert a\Vert_1=n-2\right\}$ 
\highlighttt{and remove each pattern $a\in\B^n$ with $\Vert a\Vert_1\ge n-2$ from $P_u$}.
\highlight{Even more, every insertion of a pattern $a$ into $P_{\le}$ may force some patterns to be non-feasible. If we can assume $z_D^f(E)>2$, we especially have that $1-a$ is non-feasible and can be put into $P_{>}$.} \highlighttt{Moreover, all patterns $b\succeq 1-a$ are non-feasible and can be 
removed from $P_u$.} \highlight{As remarked, $\delta\ge 1$ implies $z_D^f(E)>3$, so that $1-c$ has to 
be non-feasible whenever there are feasible patterns $a$, $b$ with $a+b=c$.} \highlighttt{Again, all patterns $a'\succeq 1-c$ are non-feasible too 
and can be removed from $P_u$.}

Because of the huge number of potential equivalence classes, the strategy of choosing  pattern $a$ from $P_u$ in line 6 of
\texttt{RecursiveProcedure} is really important. The best branching strategy we found is to choose a pattern $a \in P_u$ 
with the maximum positive multiplier in the optimal solution for the set of feasible patterns $\B^n\setminus \left\{a\in\B^n\,\:\,\exists b\in P_{>}:
\, b\ple a\right\}$. Sometimes the optimal solution has no intersection with $P_u$. In this case we can choose the branching pattern at random. 
Indeed this happens  in less than 0.01\% of all cases. This strategy reduces the \highlight{search} space of about $1000$ times in comparison 
to a random choice. 

The B\&B algorithm presented above was implemented in C++, where we used \highlight{a} self implemented LP-solver with exact arithmetic. Making
use of Intels Streaming SIMD Extensions and special shortcuts for our LP instances, our implementation of an LP-solver is about 
30 times faster than the COIN-OR LP-solver.\footnote{Cf.~\cite{kurz2012minimum}, where the author also uses a self implemented LP 
solver to enumerate the weighted simple games with $n=9$ voters.} As hardware we have used an Intel Core i7 with 4 GB RAM.

\subsection{A direct integer linear programming formulation}
\label{subsec_ilp_direct}
Instead of implementing a tailored B\&B algorithm one can also formulate the problem of the maximization of $\Delta_p(E)$ for 
a given demand $n$ as an integer programming problem and use off-the-shelf ILP solvers. To this end we describe the set $P^p(E)$ 
of feasible patterns by binary variables $y_a\in\B$ for all $a\in\B^n$ and identify $P^p(E)=\{a\in\B^n\,:\,y_a=1\}$. Partial 
information about $P^p(E)$ and $\B^n\setminus P^p(E)$ can be encoded by setting the variables of the respective patterns to either 
$1$ or $0$, respectively. Using the definition of an 1CSP instance only, we require $y_{\mathbf{0}}=1$ and $y_{e_i}=1$ for 
all $1\le i\le n$.

To ensure the existence of the parameters $L$ and $l_i$ we have to further restrict the $y_a$. Given an upper bound $M$ on 
$L$, the inequalities of Lemma~\ref{lemma_ineq} can be formulated using so-called Big-M constraints. Fortunately all this 
\highlight{is} already known in the context of weighted simple games, see \cite{kurz2012inverse,kurz2014heuristic}. So, without any further
justification we state that the 1CSP instances with demand $n$ are in one-to-one correspondence to the feasible $0/1$ solutions $y$ of:
\begin{eqnarray*}
  y_{\mathbf{0}}=1 && \\
  y_{e_i}=1 && \forall\,1\le i\le n\\
  y_{a}-y_{b}\ge 0 && \forall\,a,b\in\B^n:\, a\ple b\\
  \sum_{i:a_i=1} l_i \le L+(1-y_a)\cdot M &&\forall a\in B^n\\
  \sum_{i:a_i=1} l_i \ge L+1-y_a\cdot M &&\forall a\in B^n\\
  l_i\le l_{i+1} && \forall\, 1\le i<n\\
 \highlighttt{l_n \le L} && \\
  y_a\in\B && \forall a\in\B^n\\
  L,l_i\in\highlightt{\mathbb{Z}_{\ge 1}}&&\forall\, 1\le i\le n,
\end{eqnarray*}
where $M$ can be chosen as $4n\left(\frac{n+1}{4}\right)^{(n+1)/2}$.

In principle we would like to maximize the target function $z_D^f(E)-z_C^p(E)$. Unfortunately both terms are the optimal 
\highlightt{values} of optimization problems itself. Since the later term arises from an LP we can model optimality by using 
the duality theorem, see \cite{alpharoughly} for an application of this technique in the context of simple games. 
Here \highlightt{it} is even simpler since we can even take any feasible solution of the LP of $z_C^p(E)$ due to the maximization.
So we replace $z_C^p(E)$ by $\sum_{a\in \B^n} x_a$ and add the constraints 
\begin{eqnarray*}
  \sum_{a\in \B^n:\, a_i=1} x_a = 1 && \forall\,1\le i\le n\\
  x_a\le y_a&&\forall \,a\in\B^n\\
  x_a\in\mathbb{R}_{\ge 0} && \forall\, a\in \B^n
\end{eqnarray*}
For $z_D^f(E)$ this approach does not work, since there is no duality theorem for ILPs and $z_D^f(E)$ has a different sign as 
$z_C^p(E)$ in the target function. So we choose a different approach. By introducing 
further inequalities we can ensure that $z_D^f(E)\ge k$ holds for all feasible solutions, \highlighttt{where $k$ is an arbitrary 
but fixed integer}. 
\highlighttt{Let $y_a^i$ be additional binary variables, which equal $1$ if pattern $a\in\B^n$ can be written as the 
sum of at most $1\le i<k$ feasible patterns $a^j$, where $j=1,\dots,i$. For $i=1$ we have $y_a^1=y_a$ for all $a\in \B^n$. Next we require} 
$y_a^i\ge y_a^{i-1}$ for all $a\in \B^n$, $2\le i<k$ and
$$
  y_a^i\ge y_{u}^{i-1}+y_v^{1}-1\quad\forall \,a,u,v\in\B^n:\, u+v=a\text{ and }\forall\, 2\le i<k.
$$
\highlighttt{As a justification let us consider a pattern $a\in\B^n$ that can be written as the sum of at most $i$ feasible patterns. 
If there exists such a representation with at most $i-1$ summands,  i.e.\ $y^{i-1}_a=1$, then $y_a^i\ge y_a^{i-1}$ implies $y_a^i=1$. Otherwise there
exists a feasible pattern $v$ and a pattern $u$, that can be written as the sum of at most $i-1$ feasible patterns, with $a=u+v$. Thus, 
$y_{u}^{i-1}=1$, $y_v^{1}=1$, and so also $y_a^i=1$. We remark that $y_a^i=1$ is also possible, if pattern $a$ can not be written as 
the sum of at most $i$ feasible patterns.}

\highlighttt{With these extra variables at hand, requiring $y_{(1,\dots,1)}^{k-1}=0$ guarantees $z_D^f(E)\ge k$. Note that
a sum $\sum_{j=1}^{i} a^j\ge a$ of feasible patterns $a^j$ implies the existence of feasible patterns $\tilde{a}^j$ with
$\sum_{j=1}^{i} \tilde{a}^j= a$.}

Many of these inequalities are redundant, which is found out quickly by a customary ILP solver. We remark that it is not 
necessary to consider \highlighttt{variables $y_a^i$ for} all $i\in \{1,\dots,k-1\}$. By using inequalities of the form 
$y_a^i\ge y_{u}^{i_1}+y_v^{i_2}-1$, 
where $i_1+i_2=i$, $i_1,i_2<i$, $\Theta(\log k)$ values for $i$ are sufficient in general. Nevertheless the ILP model becomes 
quite huge, so that we solved it with the Gurobi ILP solver on an Intel Xeon with 384 GB RAM. Of course one may try to deploy
more sophisticated ILP techniques like column generation or cut separation, which goes beyond the scope of this paper.

\section{Computational results}
\label{sec_results}

For $n\le 9$ we have used the enumeration algorithm from Section~\ref{sec_enumeration} to generate all 1CSP instances with demand $n$.
In Table~\ref{tab1} we have stated the number $\left|\mathbb{P}_n^p\right|$, the maximum value $\Delta_p(E)$, the number and the corresponding
list of instances 
\highlightt{(representatives of equivalence classes))}
attaining this maximum value, whenever computationally possible. For each mentioned instance we have used 
the smallest possible integer valued parameters $l_i$ and $L$.

\setcounter{table}{0}
\begin{table}[ht]
\begin{center}
\scriptsize
\caption{Results of computational experiments
\label{tab1}}
\begin{tabular}{|c|c|c|c|c|}
\hline
m & $|\PP^b_m|$ & $\max \Delta_p$ & $^*$ & instances from classes with maximum $\Delta_p$ \\
\hline
1 & 1		& 0			& 1 & $L=1, l=(1)$ \\
\hline
2 & 2		& 0			& 2 & $L=1, l=(1,1); L=2, l=(1,1)$ \\
\hline
3 & 5		& 1/2		& 1 & $L=2, l=(1,1,1)$ \\
\hline
4 & 17		& 2/3		& 1 & $L=3, l=(1,1,1,1)$ \\
\hline
5 & 92		& 3/4		& 2 & $L=4, l=(1,1,1,1,1); L=4, l=(1,1,2,2,3)$ \\
\hline
6 & 994		& 7/8		& 1 & $L=8, l=(1, 2, 2, 3, 4, 5)$ \\
\hline
7 & 28262	& 16/17		& 1 & $L=17, l=(2, 3, 4, 5, 6, 7, 8)$ \\
\hline
8 & 2700791 & 38/39		& 1 & $L=39, l=(2, 5, 6, 8, 11, 14, 15, 18)$ \\
\hline
9 & 990331318	& 103/104	& 2 &
      $L = 104, l = (7, 12, 16, 19, 22, 27, 30, 36, 40)$ \\
 &&&& $L = 104, l = (11, 15, 18, 20, 24, 27, 28, 32, 34)$ \\
\hline
10 & & 1			& 365 &
      $L = 81, l = (4, 6, 6, 9, 16, 29, 32, 37, 40, 62)$ \\
 &&&& $L = 89, l = (4, 6, 7, 10, 18, 32, 35, 41, 44, 68)$ \\
 &&&& $L = 101, l = (5, 7, 8, 11, 20, 36, 40, 46, 50, 78)$ \\
 &&&& $L = 142, l = (7, 10, 11, 16, 28, 51, 56, 65, 70, 108)$ \\
 &&&& and 361 other instances \\
\hline
11 & & 126/125 & 6 &
      $L=155, l=(9,12,12,16,16,46,46,54,69,77,102)$ \\
 &&&& $L=193, l=(11,15,15,20,20,57,58,67,86,96,127)$ \\
 &&&& $L=204, l=(12,16,16,21,21,60,61,71,91,101,134)$ \\
 &&&& $L=207, l=(12,16,16,21,22,61,62,72,92,103,136)$ \\
 &&&& $L=218, l=(13,17,17,22,23,64,65,76,97,108,143)$ \\
 &&&& $L=221, l=(13,17,17,23,23,65,66,77,98,110,145)$ \\
\hline
12 & & 31/30 $^{**}$ & & $L=18, l=(4,4,6,6,6,7,7,9,9,9,10,12)$ \\
\hline
13 & & 53/50 $^{**}$ & &
      $L=34, l=(8,8,10,11,11,12,12,13,13,17,17,17,18)$ \\
 &&&& $L=48, l=(11,11,14,15,16,17,17,18,19,24,24,24,25)$ \\
\hline
14 & & 17/16 $^{**}$ & &
      $L=42, l=(7,7,10,10,12,15,15,21,21,21,22,22,28,31)$ \\
 &&&& $L=50, l=(8,9,12,12,14,18,18,25,25,25,26,26,33,37)$ \\
\hline
\end{tabular}
\\[2mm]
$^*$ --- number of classes with maximum $\Delta_p$

$^{**}$ --- maximum found gap, without computational proof of optimality
\end{center}
\end{table}

The stated results for $n=10,11$ are obtained with the B\&B-algorithm of Subsection~\ref{subsec_bandb}
setting $\delta$ to $1$ or \highlight{$1+\varepsilon$}\footnote{\highlight{Each $0<\varepsilon\le \frac{1}{125}$ would have worked.}}, 
respectively. The computation time for $n=11$ was 17~hours.

For the cases $12 \le n \le 14$ we restricted the search to classes with large values of $z^p_D(E)$ due to the exponential growth of 
$\left|\mathbb{P}_n^p\right|$. For $n=12$ we checked all equivalence classes with $z^p_D(E) \ge 5$, for $m=13$ only $z^p_D(E) \ge 6$, and 
for $m=14$ only $z^p_D(E) \ge 7$.

Using the ILP formulation of Subsection~\ref{subsec_ilp_direct} (and suitable bounds from Section~\ref{sec_bounds}) we have verified 
the maximum $\Delta_p$-value for $n\le 11$, while consuming a considerably larger amount of computation time.

By slightly modifying the constraints of Lemma~\ref{lemma_partial_ineq}, according to the remarks in Section~\ref{sec_simple_games}, 
we have also computed the number of weighted simple games with up to $n=9$ voters. This uncovers read-write disk failures 
within the computation done in \cite{kurz2012minimum}, so that the number of weighted simple games for $n=9$ voters was corrected from 
$989913344$ to $993061482$.

\section{Conclusion}
\label{sec_conclusion}
We have presented an enumeration algorithm for all equivalence classes of 1CSP instances. For a demand of 
at most $9$ the corresponding numbers are determined. As a side result we could correct the number of weighted simple games for $9$ 
voters (incorrectly) stated in \cite{kurz2012minimum}. To the best of our knowledge, the relation between 1CSP instances and 
weighted simple games is indicated for the first time. By enhancing the enumeration approach to a B\&B algorithm we were able 
to computationally prove that all 1CSP instances with demand of at most \highlight{$9$} are proper IRUP instances, while we found 
\highlightt{classes  of  non-IRUP instances} with demand \highlight{$n=10$ and $\Delta_p = 1$}. This resolves an open question from~\cite{friendly_bin_packing,phd_diaz}, where 
the authors ask for proper non-IRUP instances with $n<13$. \highlight{$\Delta_p>1$ is possible for $n\ge 11$ only.} Even more we have 
exactly \highlight{determined} the maximum proper gap $\Delta_p$ for $n\le 11$ 
and classified all instances attaining the maximum gap. For further investigations on the structure of 1CSP instances 
with large gap, we have made them available at \highlight{http://www.math.tu-dresden.de/{$\sim$}capad/capad.html}. By partially going 
through the search space for $n\ge 12$, we improved the worst known proper gap from $1.003$ to $1.0625$.

With respect to the exact value $z_D^p(E)$ we have proven that all 1CSP instances with $z_D^p(E)\le 3$ are proper IRUP instances
with a proper gap smaller than $1$, while there are examples with $z_D^p(E)=4$ having $\Delta_p(E)>1$.

Focusing on the size of $L$, and so indirectly on the size of the $l_i$, we mention that the first known constructions of 
proper non-IRUP instances were rather huge. The example of \cite{marcotte1986instance} has $L=3,397,386,355$ and was 
decreased to just $L=1,111,139$ in \cite{chan1998worst}. Recently the authors of \cite{friendly_bin_packing,phd_diaz} gave
an example with $L=100$. Our smallest found example has $L=18$. It would be nice to know whether this is best possible. 

We leave the famous (proper) MIRUP conjecture still widely open and encourage more research in that direction.

With respect to the enumeration of weighted simple games, the case of $n=10$ voters might be in range of the presented 
exhaustive algorithm if further tuned. Some adaptation towards the inverse power index problem, see 
\cite{kurz2012inverse,alon_edelmann_follow_up,kurz2014heuristic}, is imaginable too. 

\section*{Acknowledgements}
The auhors would like to thank two anonymous referees for very valuable remarks on an earlier draft. 
The research was supported by the Russian Foundation of Basic Research (RFBR), project 12-07-00631-a.

\nocite{rietz2002tighter}

\bibliographystyle{plain}
\bibliography{irup_cutting_stock.bib}

\end{document}